\documentclass[10pt]{amsart}

\usepackage{amssymb,latexsym,amscd,amsmath,amsthm}
\usepackage{tikz-cd}
\usepackage{xcolor}
\usepackage{geometry}
\usepackage{comment}
\usepackage{verbatim}
\usepackage{palatino}

\usepackage{hyperref}
\hypersetup{
colorlinks,
allcolors=blue
}

\usepackage[all]{xy}
\xyoption{all}
\xyoption{arc}
\xyoption{poly}

\usepackage{tikz}

\usepackage[hang,small,bf]{caption}
\newtheorem{thm}{Theorem}[section]
\newtheorem{lem}[thm]{Lemma}

\newtheorem{cor}[thm]{Corollary}
\newtheorem{prop}[thm]{Proposition}

\theoremstyle{plain}
\newtheorem{main}{Theorem}
\newtheorem{corm}[main]{Corollary}

\theoremstyle{definition}

\newtheorem{rem}[thm]{Remark}
\newtheorem*{ack}{Acknowledgments}

\newtheorem{example}{Example}[section]

\newtheoremstyle{TheoremNum}
        {\topsep}{\topsep}              
        {\itshape}                      
        {}                              
        {\bfseries}                     
        {.}                             
        { }                             
        {\thmname{#1}\thmnote{ \bfseries #3}}
    \theoremstyle{TheoremNum}

\newcommand{\R}{\mathbf{R}}
\newcommand{\Z}{{\mathbf{Z}}}

\newcommand{\C}{{\mathbf{C}}}
\newcommand{\N}{{\mathbf{N}}}
\newcommand{\cp}{\mathbf{CP}}

\newcommand{\hp}{\mathbf{HP}}
\newcommand{\sph}{\mathbf{S}}
\newcommand{\disk}{\mathbf{D}}

\newcommand{\HH}{\mathbf{H}}
\newcommand{\SU}{\operatorname{SU}}
\newcommand{\PU}{\operatorname{PU}}
\newcommand{\SO}{\operatorname{SO}}
\newcommand{\Or}{\operatorname{O}}
\newcommand{\SL}{\operatorname{SL}}

\renewcommand{\H}{\mathbf{H}}

\newcommand{\Psp}{\operatorname{PSp}}

\newcommand{\id}{\operatorname{{id}}}

\renewcommand{\lim}[1]{\mathop{\underset{#1} {\underset \longleftarrow
{\text{\rm lim}}}}}

\newcommand{\Diff}{\operatorname{Diff}}

\newcommand{\T}{\operatorname{T}}

\newcommand{\reg}{\text{reg}}
\newcommand{\triv}{\mathbf{S}^3 \times \mathbf{S}^2}
\newcommand{\ntriv}{\mathbf{S}^3 \tilde{\times} \mathbf{S}^2} 
\newcommand{\sss}{\mathbf{S}^3 \times \mathbf{S}^3}

\def\bs{\backslash}

\newcommand{\lra}{\longrightarrow}

\newcommand{\mcal}[1]{\mathcal{#1}}
\newcommand{\mc}[1]{\mathcal{#1}}

\newcommand{\ve}{\varepsilon}
\newcommand{\vphi}{\varphi}
\newcommand{\In}{\subseteq}
\def\x{\times}
\def\<{\langle}
\def\>{\rangle}

\def\bsm{\begin{smallmatrix}}
\def\esm{\end{smallmatrix}}
\def\bpm{\begin{pmatrix}}
\def\epm{\end{pmatrix}}
\def\bbm{\begin{bmatrix}}
\def\ebm{\end{bmatrix}}
\def\beq{\begin{equation}}
\def\eeq{\end{equation}}

\renewcommand{\leq}{\leqslant}

\numberwithin{equation}{section}

\newcommand{\spacing}[1]{\renewcommand{\baselinestretch}{#1}\large\normalsize}
\spacing{1.2}

\begin{document}




\title[Semi-free actions with manifold orbit spaces]{Semi-free actions with manifold orbit spaces}

\author[Harvey]{John Harvey $^{\ast}$}
\address{$^{\ast}$ Department of Mathematics, Swansea University Bay Campus, Fabian Way, Swansea, SA1~8EN, United Kingdom}
\email{john.harvey.math@gmail.com}

\author[Kerin]{Martin Kerin $^{\dagger}$}

\address{$^{\dagger}$ Mathematisches Institut, Universit\"{a}t M\"{u}nster, Einsteinstr. 62, 48149 M\"{u}nster, Germany. }
\email{m.kerin@uni-muenster.de}

\author[Shankar]{Krishnan Shankar $^{\ddagger}$}
\address{$^{\ddagger}$ Department of Mathematics, University of Oklahoma, Norman, OK 73019, USA. }
\email{Krishnan.Shankar-1@ou.edu}

\thanks{$^{\ast}$ Research carried out as part of SFB 878: \emph{Groups, Geometry \& Actions} at the University of M\"unster.}
\thanks{$^{\dagger}$ Received support from SFB 878, as well as the DFG Priority Program \emph{Geometry at Infinity}, KE 2248/1-1.}
\thanks{$^{\ddagger}$ Supported by the National Science Foundation and by SFB 878.}




\begin{abstract}
In this paper, we study smooth, semi-free actions on closed, smooth, simply connected manifolds, such that the orbit space is a smoothable manifold. We show that the only simply connected $5$-manifolds admitting a smooth, semi-free circle action with fixed-point components of codimension $4$ are connected sums of $\sph^3$-bundles over $\sph^2$. Furthermore, the Betti numbers of the $5$-manifolds and of the quotient $4$-manifolds are related by a simple formula involving the number of fixed-point components. We also investigate semi-free $S^3$ actions on simply connected $8$-manifolds with quotient a $5$-manifold and show, in particular, that there are strong restrictions on the topology of the $8$-manifold.
\end{abstract}

\date{\today}

\maketitle




\normalsize
\thispagestyle{empty}


The action of a Lie group $G$ on a smooth manifold $M$ is said to be \textit{semi-free} if, for any $p \in M$, the isotropy group at $p$ is either all of $G$ (so that $p$ is a fixed point) or trivial. Such actions have been studied extensively for the past several decades in a variety of contexts: see, for example, \cite{bredon}, \cite{churchlamotke}, \cite{levine}, \cite{my}, \cite{montyang} and, more recently, \cite{LL}, \cite{lott}, \cite{mayer}. In general, the quotient $M/G$ will not be a manifold. However, if $G = S^1$ and the fixed-point set $M^{S^1}$ has codimension four, then $M/S^1$ admits a canonical smooth structure. Similarly, if $G = S^3 \cong \SU(2)$ and the fixed-point set $M^{S^3}$ has codimension eight, then $M/S^3$ admits a canonical smooth structure.

In this article, we first investigate $5$-dimensional manifolds admitting smooth, semi-free $S^1$ actions with codimension-$4$ fixed-point sets, and then proceed to do the same for $8$-dimensional manifolds admitting smooth, semi-free $S^3$ actions with fixed-point sets of codimension eight. Throughout, we adopt the convention that the empty connected sum (that is, with zero summands) is the standard sphere.  Unless indicated otherwise, it may be assumed throughout that all manifolds are closed, smooth and simply connected, and that all actions are smooth.  The notation $\sph^d$ will distinguish the standard $d$-dimensional sphere from the spherical Lie groups $S^1$ and $S^3$.

\begin{main}
\label{T:totalspace}
	Let $M$ be a closed, smooth, simply connected, $5$-dimensional manifold and suppose that there is a smooth, semi-free circle action on $M$ so that the fixed-point set $M^{S^1}$ is of codimension four.  Then $M$ is a connected sum of $n+k-1$ $\sph^3$-bundles over $\sph^2$, where $k = b_2(M^*)$ is the second Betti number of the orbit space $M/S^1 = M^*$, and $n \in \N$ is the number of components of $M^{S^1}$.  Moreover, $M$ is spin if and only if $M^*$ (with its canonical smooth structure) is spin.
\end{main}

For the sake of clarity, we note that Theorem \ref{T:totalspace} says that a $5$-manifold $M$ admitting a semi-free circle action with codimension-$4$ fixed-point set $M^{S^1}$ is given by
$$
M = 
\begin{cases}
\#_{j = 1}^{n+k-1} (\triv) \,, & \text{ if $M$ spin,} \\
(\ntriv) \# (\#_{j = 1}^{n+k-2} (\triv)) \,, & \text{ if $M$ not spin,} \\
\end{cases}
$$
where $\ntriv$ denotes the non-trivial $\sph^3$-bundle over $\sph^2$ and $M^{S^1}$ consists of a disjoint union of $n$ embedded circles.  In Section \ref{S:examples} it will be demonstrated that all of these manifolds admit such an action, although it is difficult to pin down the diffeomorphism type of the orbit space $M^*$.  In fact, we have not ruled out the remote possibility that the actions on $\triv$ and $\ntriv$ described in Section \ref{S:examples} yield exotic smooth structures on $M^*$ (homeomorphic to $\sph^4$ and $\cp^2$, respectively).

By the work of Levine \cite{levine} (see Theorem \ref{T:levine}), it is known that all $4$-dimensional manifolds arise as the quotient of a semi-free circle action on some $5$-manifold.  It is clear from Theorem \ref{T:totalspace} that, by increasing the number of fixed-point components, each $4$-manifold may be obtained via a semi-free circle action from infinitely many different $5$-manifolds.  This is in marked contrast to the situation for principal circle bundles over $4$-manifolds, where it has been proven by Duan and Liang \cite{duanliang} that at most two (simply connected) total spaces are possible for a given base.  Moreover, they showed that the only $5$-manifolds which can admit a free circle action are connected sums of $\sph^3$-bundles over $\sph^2$.  In contrast, \cite{kollar} Koll\'{a}r has determined necessary and sufficient conditions for a $5$-manifold to admit a fixed-point-free circle action, that is, with all isotropy subgroups being finite.  In particular, many additional diffeomorphism types of $5$-manifolds can admit such an action.

In addition to these observations, notice that, if $M^*$ is a simply connected $4$-manifold and $n \in \N$, then Theorem \ref{T:totalspace} (together with Proposition \ref{P:pi1}) ensures that all semi-free actions with $n$ fixed-point components and orbit space $M^*$ occur on a fixed simply connected $5$-manifold.  On the other hand, it is well known that choosing a multiple $m \alpha$, $m \in \Z$, of a primitive $\alpha \in H^2(M^*; \Z)$ yields a principal $S^1$-bundle over $M^*$ whose total space has fundamental group $\Z_{m}$.  Thus, if $b_2(M^*) > 0$, there are infinitely many pairwise non-diffeomorphic, non-simply-connected $5$-manifolds admitting free circle actions with orbit space $M^*$.  Even if one restricts to the simply connected case $m = \pm 1$, the work of \cite{duanliang} shows that, whenever $M^*$ is not spin, there are still two possible diffeomorphism types for total spaces of the corresponding principal $S^1$-bundles over $M^*$.

Given that $5$-manifolds admitting a semi-free circle action with fixed-point set of codimension four are classified by Theorem \ref{T:totalspace}, it is natural to ask whether all such actions can also be classified up to equivariant diffeomorphism.  As it turns out, Levine \cite{levine} achieved just this in arbitrary dimensions, under certain extra hypotheses.  In the present special case, it is possible to simplify his classification.
 
\begin{main}
\label{T:classification}
{\ }
\begin{enumerate}
\item 
\label{setup}
If $M$ is a closed, smooth, simply connected, $5$-dimen\-sional manifold equipped with a smooth, semi-free circle action, with fixed-point set $M^{S^1}$ consisting of $n$ circles and orbit space $M^*$ (with its canonical smooth structure), then there is a canonical element $\bar e \in H^2(M^*; \Z)$ associated to the action.

\item 
\label{equiv}
For $i = 1,2$, let $M_i$ be a $5$-manifold as in \ref{setup}, such that $M_i^{S^1}$ consists of $n_i$ circles, and let $\bar e_i \in H^2(M_i^*; \Z)$ be the element canonically associated to the semi-free circle action.  Then the actions on $M_1$ and $M_2$ are equivariantly diffeomorphic if and only if $n_1 = n_2$ and there exists a diffeomorphism $\Psi : M_1^* \to M_2^*$ such that $\Psi^* (\bar e_2) = \bar e_1$.

\item 
\label{exist}
Given any closed, smooth, simply connected $4$-manifold $M^*$, an $\bar e \in H^2(M^*; \Z)$ and an $n \in \N$, there exists a closed, smooth, simply connected, $5$-dimen\-sional manifold $M$ on which $S^1$ acts smoothly and semi-freely with orbit space $M^*$, such that $\bar e$ is the element canonically associated to the action and the fixed-point set $M^{S^1}$ consists of $n$ circles.
\end{enumerate}
\end{main}

Notice, in particular, that it is not assumed in Theorem \ref{T:classification}\ref{equiv} that $\Psi(M_1^{S^1}) = M_2^{S^1}$.  In other words, the diffeomorphism $\Psi$ is not required to be induced by an equivariant diffeomorphism $M_1 \to M_2$. 

Theorems \ref{T:totalspace} and \ref{T:classification} continue the analysis of semi-free circle actions in low dimensions initiated by Church and Lamotke in \cite{churchlamotke}.  They proved that semi-free circle actions on $4$-manifolds with isolated fixed points are classified up to equivariant diffeomorphism by the number of fixed points.  Furthermore, only connected sums of arbitrarily many copies of $\sph^2 \x \sph^2$ can admit such actions and, as a consequence of Perelman's proof of the Poincar\'e conjecture, it turns out that the orbit space (with its canonical smooth structure) is always diffeomorphic to the standard $3$-sphere.

\medskip

We now turn our attention to the study of semi-free $S^3$ actions on $8$-manifolds with isolated fixed points.  As previously mentioned, the orbit space is a $5$-manifold and admits a canonical smooth structure.  In their article \cite{churchlamotke}, Church and Lamotke showed that all $5$-manifolds can be obtained in this way.  Furthermore, they proved that if $M_1$ and $M_2$ are two $8$-manifolds admitting semi-free $S^3$ actions with the same number $n \in \N$ of isolated fixed points and such that the corresponding orbit spaces $M_1^*$ and $M_2^*$ are diffeomorphic, then $M_1$ and $M_2$ are equivariantly homeomorphic.  

As the classification result of Church and Lamotke does not provide explicit information about the topology of those $8$-manifolds which admit semi-free $S^3$ actions with isolated fixed points, it is natural to seek a general characterization.  Although a complete classification similar to that in Theorem \ref{T:totalspace} seems difficult,  one obtains strong restrictions on a number of invariants for such spaces.

\begin{main}
\label{T:8mnfds}
Let $M$ be a simply connected, closed, smooth $8$-manifold admitting a smooth, semi-free $S^3$ action with fixed-point set $F$ consisting of $n \in \N$ isolated points.  Then the orbit space $M^*$ is a simply connected, closed $5$-manifold which admits a canonical smooth structure such that:
\begin{enumerate}
\item The integral cohomology groups of $M$ satisfy
$$
H^j(M;\Z) \cong 
\begin{cases}
H^2(M^*;\Z), & j=2,5,\\
H^3(M^*;\Z), & j=3,6,\\
\Z^{n-2}, & j=4.
\end{cases}
$$
In particular, the Betti numbers of $M$ satisfy $b_2(M) = b_3(M)$.
\item The Euler characteristic of $M$ is even and given by $\chi(M) = n$.
\item The Pontrjagin classes $p_1(M)$ and $p_2(M)$ are trivial.
\item The $\hat A$-genus and signature of $M$ are trivial.
\item $M$ is spin if and only if $M^*$ is spin.
\item The Euler class of the principal bundle $S^3 \to M\bs F \to M^* \bs F$ is a generator of $H^4(M^* \bs F; \Z)$ which pulls back to a generator of the principal $S^3$-bundle over each component of $M^* \bs F$.
\end{enumerate}
\end{main}

By taking advantage of the obstructions listed in Theorem \ref{T:8mnfds}, one can quickly rule out the existence of semi-free $S^3$ actions on a large variety of $8$-manifolds.

\begin{corm}
\label{c:nonegs}
The manifolds $\cp^4$, $\hp^2$, $\sph^2 \x \sph^6$, $\sph^3\times \sph^5$ and $\SU(3)$ do not admit a smooth, semi-free $S^3$ action with isolated fixed points.
\end{corm}

Given that the signature and cohomology of a connected sum are well understood, it is easy to use the examples in Corollary \ref{c:nonegs} to construct infinite families of $8$-manifolds which do not admit a smooth, semi-free $S^3$ action with isolated fixed points.  For example, if $M$ is one of the manifolds in Corollary \ref{c:nonegs}, then the manifold $M \# (\#_{j=1}^m (\sph^4 \x \sph^4))$ does not admit such an $S^3$ action for any $m \in \N$.

Semi-free actions by $S^1$ or $S^3$ with fixed-point sets of codimension four or eight, respectively, are special cases of a more general phenomenon.  Suppose, for example, that a Lie group $G$ acts smoothly on a closed, smooth manifold $M$ such that, locally, there are at most two orbit types.  If the action of a non-principal isotropy subgroup $H$ on the unit normal spheres to the submanifold of orbits of type $(H)$ is, up to an ineffective kernel, the Hopf action of $S^1$ on $\sph^3$ or of $S^3$ on $\sph^7$, respectively, then the orbit space $M/G$ is again a manifold and admits a canonical smooth structure.

As a simple example of this, consider the free product action of $S^1 \x S^3$ on $\sph^3 \x \sph^7$.  This induces an $S^1 \x S^3$ action as above on $\sph^{11}$, by viewing $\sph^{11}$ as the (spherical) join $\sph^3 * \sph^7$.  The orbits of type $(S^1 \x \{1\})$  form a copy of $\sph^7$ (codimension four), while the orbits of type $(\{1\} \x S^3)$ form a copy of $\sph^3$ (codimension eight), and the respective actions on the normal spheres are Hopf actions.  The orbit space, being given by $\sph^2 * \sph^4$, is therefore homeomorphic to $\sph^7$ and can be equipped with a canonical smooth structure.

All of the actions discussed in this work are special cases of fiberings with singularities, also known as Montgomery-Samelson fiberings (see, for example, \cite{Ant1}, \cite{Ant2}, \cite{CT}, \cite{CD}, \cite{Ma}, \cite{MS}, and the generalization in \cite{Hu}).  It would be interesting to know whether the results of this article have analogues in this wider setting.

Our original motivation for studying semi-free actions stems from our interest in positive sectional curvature and the work of Dyatlov \cite{Dy}, in which it was shown that positive curvature is preserved under taking quotients by semi-free actions of the types investigated here.  While Dyatlov's results have yet to yield new examples of positively curved Riemannian manifolds, we have identified strange metrics of positive curvature on $\cp^3$ by considering semi-free circle actions on positively curved Eschenburg spaces of cohomogeneity one.  

The article is organized as follows:  Basic definitions, notation and facts which will be used throughout have been collected in Section \ref{S:Prelim}.  
Section \ref{S:mayervietoris} contains the proof of Theorem \ref{T:totalspace}, while Theorem \ref{T:classification} will be proven in Section \ref{S:thmB}.  Section \ref{S:8mnfds} will deal with the proof of Theorem \ref{T:8mnfds}.  Finally, in Section \ref{S:examples} we give some examples and constructions of semi-free actions with manifold orbit spaces.  

\begin{ack} 
M.\ Kerin would like to thank F.\ Galaz-Garc\'ia for several helpful conversations, while all three authors are grateful to L.\ Kennard for useful comments on an initial draft of the paper.
\end{ack}


\section{Basic definitions and notation}
\label{S:Prelim}

As mentioned in the introduction, the standard $n$-dimensional sphere will be denoted by $\sph^n$ and the non-trivial $\sph^3$-bundle over $\sph^2$ by $\ntriv$, while the spherical Lie groups will be denoted by $S^1$ and $S^3$.  The $n$-dimensional open (unit) disk will be $\disk^n$.

Recall that a smooth, oriented $n$-dimensional manifold $M$ is said to be \emph{spin} if the second Stiefel--Whitney class $w_2(M) \in H^2(M;\Z_2)$ vanishes and that being spin is an invariant of the homotopy type of a closed, smoothable manifold \cite{thom}.  By Poincar\'{e} duality, a closed, simply connected $4$-manifold $M$ always has torsion-free integral cohomology, with the rank of $H^2(M; \Z)$ being given by the second Betti number $b_2(M)$.  
Furthermore, as a consequence of the Barden--Smale diffeomorphism classification \cite{barden, smale}, if a closed, smooth, simply connected $5$-manifold $M$ has $H_2(M; \Z)$ torsion free of rank $k = b_2(M)$, then $M \cong \#_{j= 1}^k (\triv)$, whenever $M$ is spin, that is, whenever $w_2(M)=0$, while $M \cong (\ntriv) \# (\#_{j = 1}^{k-1} (\triv))$, whenever $w_2(M) \neq 0$.

A \emph{smooth action} of a compact Lie group $G$ on a smooth manifold $M$ is a smooth map $G \x M \to M \,;\, (g,p) \mapsto g \cdot p$ such that $G \to \Diff(M) \,;\, g \mapsto (p \mapsto g \cdot p)$ is a homomorphism.  The \emph{isotropy subgroup} of an action at $p \in M$ is $G_p := \{g \in G \mid g \cdot p = p\} \In G$, while the submanifold $G \cdot p := \{ q \in M \mid q = g \cdot p, \, g \in G\} \In M$ is called the \emph{orbit} through $p \in M$ and is diffeomorphic to the quotient $G / G_p$.  An orbit $G \cdot p$ is said to be of \emph{type} $(H)$ if $G_p$ is conjugate to a given closed subgroup $H$ in $G$.  The \emph{orbit space} $M/G$ of an action will be denoted by $M^*$, leading to the notation $p^*$ and $A^*$ for the image under the orbit projection map $\pi: M \to M^*$ of a point $p \in M$ and a set $A \In M$, respectively.

A point $p \in M$ is called a \emph{fixed point} of an action if $G_p = G$.  An action will be called \emph{semi-free} if there are only two orbit types: fixed points and orbits with trivial isotropy, that is, $G_p = G$ or $G_p = \{e\}$ for all $p \in M$.  An orbit of type $(\{e\})$ will be called \emph{principal}.

The set $M^G \In M$ of all fixed points of an action will be denoted by $F$, whenever there is no confusion.  In particular, $F$ is a smooth submanifold of $M$ with image $\pi(F) = F^* \In M^*$ diffeomorphic to $F$.  For this reason, $F$ will usually also be used to refer to $F^*$, and it will be clear from the context whether it is being viewed as a subset of $M$ or $M^*$.

If compact Lie groups $G$ and $G'$ act smoothly on manifolds $M$ and $N$, respectively, and $\vphi : G \to G'$ is a Lie group homomorphism, then a smooth map $f : M \to N$ is said to be \emph{equivariant with respect to} $\vphi$ if $f(g \cdot p) = \vphi(g) \cdot f(p)$ for all $g \in G$ and $p \in M$.  In particular, if $f : M \to N$ is a diffeomorphism and $\vphi: G \to G'$ is an isomorphism, then $f$ is said to be an \emph{equivariant diffeomorphism}, while $M$ and $N$ are said to be \emph{equivariantly diffeomorphic}.  In the special case that $\vphi = \text{id}_G$, the actions will be called \emph{equivalent}.

Suppose now that $S^r$, $r \in \{1, 3\}$, acts smoothly and semi-freely on a smooth manifold $M$.  As the action is free away from the fixed-point set $F \In M$ and the components of $F$ are smooth submanifolds of even codimension, it follows that the orbit space $M^*$ is stratified into smooth submanifolds given by the codimension-zero regular part $M^*_\reg := (M \bs F)/S^r = M^* \bs F$, together with the odd-codimension components of $F = F^*$.  In particular, the link, that is, the space of normal directions, of each component of $F \In M^*$ is either a complex or quaternionic projective space, depending on whether $r = 1$ or $r = 3$.  Although the regular part $M^*_\reg$ inherits a smooth structure from that of $M$, the same is not true for all of $M^*$, in general.

However, if all components of $F \In M$ are of codimension $2(r + 1)$, the corresponding links in $M^*$ are given by $\cp^1 \cong \sph^2$, if $r = 1$, and by $\hp^1 \cong \sph^4$, if $r = 3$.  In this special case, $M^*$ is a topological manifold and, moreover, the smooth structure on $M^*_\reg$ can be extended in a canonical manner to all of $M^*$, yielding the \emph{canonical smooth structure} on $M^*$.  In brief, if $r = 1$ (resp.\ $r = 3$), there is a complex (resp.\ quaternionic) structure on the normal bundle to each component of $F \In M$, and this induces a $\PU(2)$-bundle (resp.\ $\Psp(2)$-bundle) structure on the normal bundle to components of $F \In M^*$.  As the actions of $\PU(2)$ on $\cp^1$ and $\Psp(2)$ on $\hp^1$ are equivariantly diffeomorphic to those of $\SO(3)$ on $\sph^2$ and $\SO(5)$ on $\sph^4$, respectively, there is a canonical way to smooth the orbit space, further details of which can be found in \cite{churchlamotke} and \cite{levine}.  

Throughout the discussions below, it will often be assumed without comment that manifolds are closed, smooth and simply connected and that actions are smooth.


\section{Topology of the \texorpdfstring{$5$}{5}-manifolds}
\label{S:mayervietoris}

The existence of a semi-free circle action on a manifold of any dimension places strong restrictions on the fundamental group, since any orbit can easily be contracted to a point by making use of a path to a fixed point.

\begin{prop}\label{P:pi1}
	Let $M$ be a closed, smooth manifold admitting a semi-free circle action with fixed-point set $F$ of codimension $4$ and orbit space $M^*$. Then $M$ is simply connected if and only if $M^*$ is simply connected.
\end{prop}

\begin{proof}
	If $M$ is simply connected then $M^*$ is simply connected, by, for example, the path-lifting result of Montgomery and Yang \cite[Cor.\ 2]{my}.
	
	Conversely, suppose that $M^*$ is simply connected and choose a base point $p^* \in M^*$ corresponding to a principal orbit $S^1 \cdot p \In M$.  Consider two loops in $M ^* \bs F \In M^*$ based at $p^*$ and a homotopy in $M^*$ between them.  As the codimension of $F \In M$ is $4$, the orbit space $M^*$ is a manifold and $F \In M^*$ is a codimension-$3$ submanifold.  By transversality, the homotopy can be made to lie entirely in $M^* \bs F$, which implies that $M^* \bs F$ is also simply connected. 
	
	Now, since $M \bs F$ is the total space of a principal $S^1$-bundle over a simply connected base, its fundamental group is generated by the orbit $S^1 \cdot p$. 	Any loop in $M$ based at $p$ can be homotoped into $M \bs F$ by transversality and, hence, to a multiple of the orbit $S^1 \cdot p$.  Finally, if $c: [0,1] \to M$ is a curve with $c(0) = p$ and $c(1) \in F$, that is, with $c(1)$ a fixed point, then the set $\{S^1 \cdot c(t) \mid t \in [0,1]\}$ describes a $2$-disk in $M$ along which the orbit $S^1 \cdot p$ can be contracted to a point.  Therefore, $M$ is simply connected.
\end{proof}

To begin the proof of Theorem \ref{T:totalspace}, the topology of the orbit space and the number of fixed-point components will be used to garner information about the possible diffeomorphism types for a $5$-manifold which admits a semi-free circle action with fixed-point set of codimension four.  In Section \ref{S:examples} it is demonstrated that all of these diffeomorphism types admit such an action.

\begin{thm}
	Let $M$ be a closed, smooth 5-manifold admitting a semi-free circle action with fixed-point set $F$ consisting of $n$ circles.  If either $M$ or the orbit space $M^*$ is simply connected, then $M$ is a connected sum of $n+k-1$ $\sph^3$-bundles over $\sph^2$, where $b_2(M^*)=k$.
\end{thm}

\begin{proof}
	By Proposition \ref{P:pi1}, both $M^*$ and $M$ are simply connected.  Therefore, by \cite{barden} and Poincar\'{e} duality, it suffices to show that $H^3(M; \Z) \cong H_2(M; \Z) \cong \Z^{n+k-1}$.
	
	Since $M^*$ is a simply connected $4$-manifold, $H^2(M^*; \Z)$ is torsion-free of rank $k = b_2(M^*)$ and $H^3(M^*; \Z)$ is trivial.  If $D(F) \In M$ denotes a (sufficiently small) $S^1$-invariant neighbourhood of $F \In M$ and $D^*(F) \In M^*$ its image in $M^*$, the Mayer--Vietoris sequence for $M^* = (M^* \bs F) \cup D^*(F)$ yields $H^2(M^* \bs F; \Z) \cong \Z^{n+k}$ and $H^3(M^* \bs F; \Z) \cong \Z^{n-1}$, since the intersection $(M^* \bs F) \cap D^*(F)$ is homotopy equivalent to $F \x \sph^2$, that is, to a disjoint union of $n$ copies of $\sph^1 \x \sph^2$.
	
Now $M \bs F$ is the total space of a principal $S^1$-bundle over $M^* \bs F$ and, by the proof of Proposition \ref{P:pi1}, $M^*\bs F$ is simply connected.	From the corresponding Gysin sequence, one obtains the short exact sequence 
$$
0 \to \Z^{n-1} \to H^3(M \bs F; \Z) \to \Z^{k+n} \to 0 \,,
$$
from which it follows that $H^3(M \bs F; \Z) \cong \Z^{2n+k-1}$.
	
As $M$ is simply connected, Poincar\'{e} duality gives $H^4(M; \Z) \cong 0$.  On the other hand, $M = (M \bs F) \cup D(F)$ and the intersection $(M \bs F) \cap D(F)$ is homotopy equivalent to $F \x \sph^3$, that is, to a disjoint union of $n$ copies of $\sph^1 \x \sph^3$.  Therefore, the Mayer--Vietoris sequence yields the short exact sequence 
$$
0 \to H^3(M; \Z) \to \Z^{2n+k-1} \to \Z^n \to 0\,,
$$ 
from which it may be concluded that $H^3(M; \Z) \cong \Z^{n+k-1}$, as desired.
\end{proof}

To complete the proof of Theorem \ref{T:totalspace}, it remains to determine the behaviour of spin structures under semi-free actions, which is quite different to that for free actions.

\begin{thm}
	Let $M$ be a closed, smooth, simply connected $5$-manifold admitting a semi-free circle action with fixed-point set $F$ consisting of $n$ circles.  Then $M$ is spin if and only if the orbit space $M^*$ is spin.
\end{thm}

\begin{proof}
Let $\hat \jmath : M \bs F \to M$ and $j : M^* \bs F \to M^*$ denote the respective inclusion maps.  Then, using the same notation as in the previous proof, the Mayer--Vietoris sequence with $\Z_2$ coefficients for $M = (M \bs F) \cup D(F)$ shows that $\hat \jmath^* : H^2(M; \Z_2) \to H^2(M \bs F; \Z_2) $ is an isomorphism, while that for $M^* = (M^* \bs F) \cup D^*(F)$ yields that $j^* : H^2(M^*; \Z_2) \to H^2(M^* \bs F; \Z_2) $ is injective.  On the other hand, the tangent bundles of $M \bs F$ and $M^* \bs F$ are the pullbacks of the tangent bundles of $M$ and $M^*$ under the respective inclusions.  From $w_2(M \bs F) = \hat \jmath^* w_2(M)$ and $w_2(M^* \bs F) = j^* w_2(M^*)$ it then follows that $M \bs F$ is spin if and only if $M$ is spin, and $M^* \bs F$ is spin if and only if $M^*$ is spin.

The tangent bundle of $M \bs F$ is given by $T(M \bs F) = \mc V \oplus \pi^* T(M^* \bs F)$, where $\mc V$ is the vertical distribution of the principal bundle $S^1 \to M \bs F \stackrel{\pi}{\lra} M^* \bs F$.  As $\mc V$ is parallelizable, the Whitney sun formula yields $w_2(M \bs F) = \pi^* w_2( M^* \bs F)$.

Now, from the Gysin sequence for $S^1 \to M \bs F \stackrel{\pi}{\lra} M^* \bs F$ it is clear that the kernel of $\pi^* : H^2(M^* \bs F; \Z_2) \to H^2(M \bs F; \Z_2)$ is isomorphic to $\Z_2$ and generated by the mod-$2$ reduction of the Euler class $e(\pi)$ of the bundle, that is, by the second Stiefel--Whitney class $w_2(\pi) \in H^2(M^* \bs F; \Z_2)$.  Therefore, it suffices to show that, if $w_2(M^* \bs F) \neq 0$, then $w_2(M^* \bs F) \neq w_2( \pi)$.

To this end, suppose that $w_2(M^* \bs F) \neq 0$ and let $i : D^*(F) \bs F \to M^* \bs F$ denote the inclusion map, where $D^*(F)\bs F =  (M^* \bs F) \cap D^*(F)$.  Exactness of the Mayer--Vietoris sequence for $M^* = (M^* \bs F) \cup D^*(F)$ now yields
$$
i^* w_2(M^*\bs F) = i^*(j^* w_2(M^*)) = 0 \in H^2(D^*(F) \bs F; \Z_2).
$$
On the other hand, as the pullback $S^1 \to D(F) \bs F \stackrel{i^* \pi}{\lra} D^*(F) \bs F$ of the principal bundle $S^1 \to M \bs F \stackrel{\pi}{\lra} M^* \bs F$ is homotopy equivalent to the disjoint union of $n$ copies of the ($\id_{\sph^1} \x$ Hopf)-fibration $\sph^1 \x \sph^3 \to \sph^1 \x \sph^2$, it follows that the image $i^* e(\pi) = e(i^* \pi) \in H^2(D^*(F) \bs F; \Z) \cong \Z^n$ of the Euler class $e(\pi) \in H^2(M^* \bs F; \Z)$ is a generator in each summand.  In particular, reducing mod-$2$ yields
$$
i^* w_2(\pi) = w_2(i ^* \pi) \neq 0 \in H^2(D^*(F) \bs F; \Z_2),
$$
from which it immediately follows that $w_2(M^* \bs F) \neq w_2( \pi)$ and, hence, that $w_2(M \bs F) = \pi^* w_2(M^* \bs F) \neq 0$, as desired.
\end{proof}

As is also seen in Section \ref{S:examples}, the spin structure, or lack thereof, affects the number of fixed-point components of a semi-free circle action.

\begin{cor}
Suppose the fixed-point set $F$ of a semi-free circle action on a closed, smooth, simply connected $5$-manifold $M$ has codimension four and consists of $n$ components.  Then $n \in \{ b_2(M) + 1 - 2j \mid 0 \leq j \leq \lfloor \frac{b_2(M)}{2} \rfloor \}$, if $M$ is spin, while $n \in \{1, \dots, b_2(M)\}$, if $M$ is not spin.  Moreover, all possible $n$ are achieved.
\end{cor}

\begin{proof}
By Theorem \ref{T:totalspace}, we have $b_2(M) = n+b_2(M^*)-1$.  Suppose $M$ is not spin.  Then, since a non-spin $4$-manifold has positive second Betti number, it follows that $n \leq b_2(M)$.  By Proposition \ref{P:connsum}, $M$ admits a semi-free circle action with fixed-point set consisting of $n$ circles, for all $n \in \{1, \dots, b_2(M)\}$.

Suppose now that $M$ is spin.  By Rokhlin's Theorem \cite{rokhlin}, the signature of the intersection form of $M^*$ is divisible by $16$.  Therefore, $b_2(M^*)$ must also be even.  By Propositions \ref{P:connsum} and \ref{P:fibresums}, all $b_2(M^*) = 2j$ with $j \in \{0 \leq j \leq \lfloor \frac{b_2(M)}{2} \rfloor\}$ are achieved, from which the claim follows.
\end{proof}


\section{A refined classification}
\label{S:thmB}

Under the hypothesis that the fixed-point set has no $2$-torsion in its second cohomology group, Levine \cite{levine} has given a classification of semi-free circle actions with fixed-point sets of codimension four.  As the fixed-point set of such an action in dimension five is easily understood, this classification can be simplified somewhat.

Recall, first, that a free circle action on a manifold $M$ with orbit space $M^*$ is classified by the Euler class, that is, by an element of $H^2(M^*; \Z)$.  Suppose, for $i = 1,2$, that $M_i$ admits a free circle action with orbit space $M_i^*$ and Euler class $e_i \in H^2(M_i^*; \Z)$, respectively.  Then these actions are equivalent if and only if there is a diffeo\-morphism $\vphi : M_1^* \to M_2^*$ such that $\vphi^*(e_2) = e_1$.  This is illustrated by the following simple examples.

\begin{example} {\ }
\begin{enumerate}
\item  
\label{hopf}
$\sph^3$ admits two free circle actions with orbit space $\sph^2$, having Euler classes $\pm 1 \in H^2(\sph^2; \Z)$, respectively.  As the antipodal map on $\sph^2$ interchanges these two classes, the actions are equivalent.

\item $\sph^3 \x \sph^3$ admits free circle actions with orbit spaces $\triv$ and $\sph^2 \x \sph^3$ and Euler classes $(1, 0) \in H^2( \triv; \Z)$ and $(0, 1) \in H^2( \sph^2 \x \sph^3; \Z)$, respectively. The diffeomorphism $\triv \to \sph^2 \x \sph^3$ which interchanges the factors shows that the actions on $\sph^3 \x \sph^3$ are equivalent.

\item Each of $\triv$ and $\sph^2 \x \sph^3$ admits a free circle action with orbit space $\sph^2 \x \sph^2$, such that the Euler classes are $(1,0), (0,1) \in H^2(\sph^2 \x \sph^2; \Z)$, respectively. The diffeomorphism of $\sph^2 \x \sph^2$ which interchanges the factors shows that these actions are equivalent.
\end{enumerate}
\end{example}

Levine's classification result for semi-free circle actions with fixed-point set of co\-dimension four has a similar flavour.

\begin{thm}[Levine \cite{levine}]
\label{T:levine}
{\ }

\begin{enumerate}
\item Suppose that $S^1$ acts smoothly and semi-freely on a smooth, closed manifold $M$ with orbit space $M^*$, such that the fixed-point set $F \In M$ has codimension four and $H^2(F; \Z)$ is $2$-torsion free.  Let $i: \sph^2 \to M^* \bs F$ be the inclusion of a fibre of the normal sphere-bundle of $F \In M^*$and let $e \in H^2(M^* \bs F; \Z)$ be the Euler class of the principal bundle $S^1 \to M \bs F \to M^* \bs F$.  Then $H^2(\sph^2; \Z)$ is generated by $i^* e$ and, furthermore, the $S^1$ action on $M$ is determined up to equivariant diffeomorphism by the diffeomorphism type of $(M^*, F, e)$.
\item Suppose a closed, smooth manifold $M^*$ has a submanifold $F \In M^*$ of codimension three, with $H^2(F; \Z)$ $2$-torsion free and having oriented normal bundle.  Suppose further that there is an element $e \in H^2 (M^* \bs F; \Z)$ such that the pullback $i^* e$ under any inclusion $i: \sph^2 \to M^* \bs F$ of a fibre of the normal sphere-bundle of $F \In M^*$ is a generator of $H^2(\sph^2; \Z)$.  Then there exists a closed, smooth manifold $M$ admitting a semi-free $S^1$ action with orbit space $M^*$, fixed-point set $F \In M$, and such that the principal bundle $S^1 \to M \bs F \to M^* \bs F$ has Euler class $e$.
\end{enumerate}
\end{thm}

The present goal is to reduce Theorem \ref{T:levine} to Theorem \ref{T:classification} in dimension $5$.  The following lemmas will facilitate this.

\begin{lem}
\label{L:difftype}
For $i = 1,2$, let $M_i^*$ be a closed, smooth,  oriented $4$-manifold and $F_i \In M_i^*$ a disjoint union of $n_i$ embedded circles.  Then $M^*_1 \bs F_1$ is diffeomorphic to $M^*_2 \bs F_2$ if and only if $n_1 = n_2$ and there is a diffeomorphism $\Psi : M^*_1 \to M^*_2$. 
\end{lem}

\begin{proof}
By the work of Hatcher \cite{hatcher}, the diffeomorphism group of the $3$-disk $\disk^3$ is homotopy equivalent to $\Or(3)$.  Therefore, every oriented $\disk^3$-bundle over $\sph^1$ must be trivial.  In particular, each component of the (oriented) normal disk-bundle $D^*(F_i)$ to $F_i \In M^*_i$ is trivial.  Hence, for each $i = 1,2$, there is a diffeomorphism 
\beq
\label{E:beta}
\beta_i : \coprod_{j=1}^{n_i} (\sph^1 \x \sph^2) \cong \partial D^*(F_i)  
\to \partial(M_i^* \bs D^*(F_i))
\eeq
such that $M_i^* \cong (M_i^* \bs D^*(F_i)) \cup_{\beta_i} \coprod_{j=1}^{n_i} (\sph^1 \x \disk^3)$.

Suppose now that there is a diffeomorphism $\vphi : M_1^* \bs F_1 \to M_2^* \bs F_2$.  In particular, this forces $n_1 = n_2 =: n$, since $\vphi \circ \beta_1 = \beta_2$.  Moreover, by \cite[Chap.\ 8, Thm 2.2]{Hi}, the fact that the diffeomorphism
$$
\beta_2 \circ \beta_1^{-1} = \vphi|_{\partial(M_1^* \bs D^*(F_1))} : \partial(M_1^* \bs D^*(F_1)) \to \partial(M_2^* \bs D^*(F_2))
$$
extends to a diffeomorphism $M_1^* \bs D^*(F_1) \to M_2^* \bs D^*(F_2)$ (namely, $\vphi$ itself) implies that there is a diffeomorphism 
\beq
\label{E:Phi}
\Phi : M_1^* \cong (M_1^* \bs D^*(F_1)) \cup_{\beta_1} \coprod_{j=1}^{n} (\sph^1 \x \disk^3) 
\to (M_2^* \bs D^*(F_2)) \cup_{\beta_2} \coprod_{j=1}^{n} (\sph^1 \x \disk^3) \cong M_2^*
\eeq
which restricts to $\vphi$ on $M_1^* \bs D^*(F_1)$ and to the identity map on $\coprod_{j=1}^{n} (\sph^1 \x \disk^3)$.  In addition, taking into account the identifications $D^*(F_i) \cong  \coprod_{j=1}^{n} (\sph^1 \x \disk^3)$, $i=1,2$, it is immediate that $\Phi(F_1) = F_2$.

Suppose, on the other hand, that $\Psi : M_1^* \to M_2^*$ is a diffeomorphism and $n_1 = n_2$.  By the work of Haefliger \cite{haef} and Chazin \cite{chaz}, there is a unique isotopy class of embeddings of $\coprod_{j=1}^{n_1} \sph^1$ into $M_1^*$.  In particular, $F_1$ and $\Psi^{-1}(F_2)$ are isotopic submanifolds of $M_1^*$.  Hence, by \cite[Chap.\ 8, Thm 1.3]{Hi}, there is an ambient isotopy $h: M_1^* \x [0,1] \to M_1^*$ such that $h_0 = \id_{M_1^*}$ and $h_1(F_1) = \Psi^{-1}(F_2)$.  Therefore, the diffeomorphism $\Psi$ is isotopic to a diffeomorphism $\Phi := \Psi \circ h_1 : M_1^* \to M_2^*$ satisfying $\Phi(F_1) = F_2$.  In particular, the restriction $\Phi|_{M^*_1 \bs F_1} : M^*_1 \bs F_1 \to M^*_2 \bs F_2$ is a diffeomorphism.
\end{proof}

Suppose that $M$ is a $5$-manifold on which $S^1$ acts semi-freely with orbit space $M^*$ and fixed-point set $F \In M$ consisting of $n$ circles.  Since $M$ is simply connected, so too are $M^*$ and $M^* \bs F$, by Proposition \ref{P:pi1} (and its proof).  By Poincar\'e duality, $H^*(M^*; \Z)$ is therefore torsion free.  From the Mayer--Vietoris sequence for $M^* = (M^* \bs F) \cup D^*(F)$ there is a short exact sequence
\beq
\label{E:MVH2}
0 \to H^2(M^*; \Z) \stackrel{j^*}{\lra} H^2(M^* \bs F; \Z) \stackrel{\eta^*}{\lra} H^2(S^*(F); \Z) \to 0 \,,
\eeq
where $j : M^* \bs F \to M^*$ and $\eta: S^*(F) \to M^* \bs F$ are the respective inclusion maps and $S^*(F) \cong F \x \sph^2 \In (M^* \bs F) \cap D^*(F)$ is the normal sphere-bundle to $F \In M^*$.  In particular, $H^2(S^*(F); \Z) \cong \Z^n$ and, hence, the sequence splits.

\begin{lem}
\label{L:section}
The short exact sequence \eqref{E:MVH2} admits a canonical section $\tau: H^2(M^*\bs F; \Z) \to H^2(M^*; \Z)$ with $\tau \circ j^* = \id_{H^2(M^*; \Z)}$.
\end{lem}

\begin{proof}
Let $N \In M^*$ be the union of $n$ mutually disjoint, embedded $4$-disks and let $F' \In N$ be the union of $n$ embedded circles such that each component of $N$ contains exactly one component of $F'$ (in its interior).  Observe, in particular, via a simple Mayer--Vietoris argument, that the inclusion $\rho : M^* \bs N \to M^*$ induces an isomorphism $\rho^* : H^2(M^*; \Z) \to H^2(M^* \bs N; \Z)$.

By the results of Haefliger \cite{haef} and Chazin \cite{chaz}, together with \cite[Chap.\ 8, Thm 1.3]{Hi}, there is an ambient isotopy $h : M^* \x [0,1] \to M^*$ such that $h_0 = \id_{M^*}$ and $h_1 (F) = F'$.  Moreover, being isotopic to the identity map, it is clear that the diffeomorphism $h_1$ induces the identity homomorphism on $H^2(M^*;Z)$.

Consider the commutative diagram
$$
\xymatrix{
 & M^* \bs F' \ar[r]^{h_1^{-1}} & M^* \bs F  \ar[dd]^j \\
M^* \bs N \ar[ur]^\theta \ar[dr]^{\rho} & & \\
 & M^* \ar[r]^{h_1^{-1}} & M^*
}
$$
where $\theta : M^* \bs N \to M^* \bs F'$ is the inclusion map.  The section $\tau$ can now be defined via
\beq
\label{E:section}
\tau := ((h_1^{-1} \circ \rho)^*)^{-1} \circ (h_1^{-1} \circ \theta)^* : H^2(M^* \bs F; \Z) \to H^2(M^*;\Z).
\eeq
In particular, it follows that 
$$
\tau =  (h_1^{-1})^* \circ (\rho^*)^{-1} \circ \theta^* \circ (h_1^{-1}|_{M^*\bs F'})^* = (\rho^*)^{-1} \circ \theta^* \circ (h_1^{-1}|_{M^*\bs F'})^*
$$ 
and, hence, that $\tau \circ j^* = \id_{H^2(M^*;\Z)}$ as desired.

The section $\tau$ is canonical in the sense that it depends only on $M^*$ and $F$ and not on the choice of either $N$ or $F'$.  Indeed, by \cite[Chap.\ 8, Thm.\ 3.2]{Hi} there is an ambient isotopy of $M^*$ which isotops $N$ to any other collection of $n$ disjoint, embedded disks.  Similarly, there is an ambient isotopy of $M^*$ deforming any other collection $\tilde F' \In N$ of $n$ embedded circles as above to $F'$ by first deforming to $F$ and then proceeding as before.
\end{proof}

It turns out that the canonical section behaves well in scenarios typical of equivariant actions.

\begin{lem}
\label{L:commute}
Suppose that, for $i = 1,2$, $M_i$ is a simply connected $5$-manifold on which $S^1$ acts semi-freely with orbit space $M_i^*$ and fixed-point set $F_i \In M_i$ consisting of $n_i$ circles.  Suppose further that $\Phi : M_1^* \to M_2^*$ is a diffeomorphism such that $\Phi(F_1) = F_2$ and $\vphi := \Phi|_{M_1^* \bs F_1}$.  Then 
$$
\tau_1 \circ \vphi^* = \Phi^* \circ \tau_2
$$
where $\tau_i : H^2(M_i^* \bs F_i ; \Z) \to H^2(M_i^*; \Z)$, $i = 1,2$, are the canonical sections from Lemma \ref{L:section}.
\end{lem}

\begin{proof}
Since the claim regards induced maps in cohomology, it may be assumed without loss of generality (by choosing appropriate ambient isotopies of $M_1^*$ and $M_2^*$ as in the proof of Lemma \ref{L:section}) that there are embedded disjoint unions $N_i \In M_i^*$ of $ n_i$ $4$-disks, $i = 1,2$, such that $N_2 = \Phi(N_1)$, $F_2 = \Phi(F_1)$, and each component of $N_i$ contains exactly one component of $F_i$.  In particular, $n_1 = n_2$.

In this setup, the inclusion maps $\rho_i : M_i^* \bs N_i \to M_i^*$ and $\theta_i : M_i^* \bs N_i \to M_i^* \bs F_i$, $i = 1,2$, satisfy $\vphi \circ \theta_1 = \theta_2 \circ \Phi|_{M_1^* \bs N_1}$ and $\Phi \circ \rho_1 = \rho_2 \circ \Phi|_{M_1^* \bs N_1}$, while the canonical sections defined by \eqref{E:section} are given by $\tau_i = (\rho_i^*)^{-1} \circ \theta_i^*$, $i = 1,2$.

It now follows that
\begin{align*}
\tau_1 \circ \vphi^* &=  (\rho_1^*)^{-1} \circ \theta_1^* \circ \vphi^* \\
&= (\rho_1^*)^{-1} \circ (\Phi|_{M_1^* \bs N_1})^* \circ \theta_2^* \\
&= \Phi^* \circ (\rho_2^*)^{-1} \circ \theta_2^* \\
&= \Phi^* \circ \tau_2
\end{align*}
as desired.
\end{proof}

Now that all the ingredients are in place, it is possible to prove Theorem \ref{T:classification}.

\begin{proof}[Proof of Theorem \ref{T:classification}]
Suppose that $M$ is a $5$-manifold on which $S^1$ acts semi-freely with orbit space $M^*$ and fixed-point set $F \In M$ consisting of $n$ circles.  Since $M$ is simply connected, so too are $M^*$ and $M^* \bs F$, by Proposition \ref{P:pi1} (and its proof).  
As in \eqref{E:MVH2}, the short exact sequence 
$$
0 \to H^2(M^*; \Z) \stackrel{j^*}{\lra} H^2(M^* \bs F; \Z) \stackrel{\eta^*}{\lra} H^2(S^*(F); \Z) \to 0 \,,
$$
splits, where $j : M^* \bs F \to M^*$ and $\eta: S^*(F) \to M^* \bs F$ are the respective inclusion maps and $S^*(F) \cong F \x \sph^2 \In (M^* \bs F) \cap D^*(F)$ is the normal sphere-bundle to $F \In M^*$.  

Let $\tau: H^2(M^*\bs F; \Z) \to H^2(M^*; \Z)$ be the canonical section given by Lemma \ref{L:section}.  If $e \in H^2(M^* \bs F ; \Z)$ is the Euler class of the principal bundle $S^1 \to M \bs F \to M^* \bs F$, then there is a canonical element $\bar e := \tau(e) \in H^2(M^*; \Z)$ associated to the action, as asserted in Theorem \ref{T:classification}\ref{setup}.

In order to prove Theorem \ref{T:classification}\ref{equiv}, suppose that, for $i = 1,2$, $M_i$ is a $5$-manifold on which $S^1$ acts semi-freely with orbit space $M_i^*$ and fixed-point set $F_i \In M_i$ consisting of $n_i$ circles.  By Theorem \ref{T:levine}, these actions are equivariantly diffeomorphic if and only if $n_1 = n_2$ (since $F_1$ must be diffeomorphic to $F_2$) and there is a diffeomorphism $\vphi : M_1^* \bs F_1 \to M_2^* \bs F_2$ such that the Euler classes $e_i \in H^2(M_i^* \bs F_i ; \Z)$, $i = 1, 2$, of the respective principal bundles $S^1 \to M_i \bs F_i \to M_i^* \bs F_i$, $i = 1, 2$, satisfy $\vphi^*(e_2) = e_1$. 

In this case, by Lemma \ref{L:difftype}, $\vphi$ can be extended to a diffeomorphism $\Phi : M_1^* \to M_2^*$ with $\Phi (F_1) = F_2$ and such that there is a commutative diagram (cf.\ \eqref{E:MVH2})
\beq
\label{E:diag}
\xymatrix{
0 \ar[r] & H^2(M^*_1; \Z) \ar[r]^(0.45){j_1^*} & H^2(M_1^* \bs F_1; \Z) \ar[r]^{\eta_1^*} \ar@/^1pc/@{.>}[l]|{\tau_1} & H^2(S^*(F_1); \Z) \ar[r] & 0 \\
0 \ar[r] & H^2(M^*_2; \Z) \ar[r]^(0.45){j_2^*} \ar[u]^{\Phi^*} & H^2(M_2^* \bs F_2; \Z) \ar[r]^{\eta_2^*} \ar[u]^{\vphi^*} \ar@/^1pc/@{.>}[l]|{\tau_2} & H^2(S^*(F_2); \Z) \ar[r] \ar[u]^{\vphi^*} & 0 \,.
}
\eeq
In particular, from Lemma \ref{L:commute} it follows that 
$$
\bar e_1 = \tau_1(e_1) = \tau_1(\vphi^*(e_2)) = \Phi^*(\tau_2^*(e_2)) = \Phi^*(\bar e_2)
$$
as desired.

Conversely, suppose that $n_1 = n_2 =: n$ and there is a diffeomorphism $\Psi : M_1^* \to M_2^*$ such that $\Psi^*(\bar e_2) = \bar e_1$.  By \cite{haef} and \cite{chaz}, there is an ambient isotopy $h : [0,1] \x M_1^* \to M_1^*$ such that $h_0 = \id_{M_1^*}$ and $h_1(F_1) = \Psi^{-1}(F_2)$.  Therefore, $\Phi := \Psi \circ h_1 : M_1^* \to M_2^*$ is a diffeomorphism satisfying $\Phi(F_1) = F_2$ and $\Phi^* = \Psi^* : H^2(M_2^* ; \Z) \to H^2(M_1^*; \Z)$.  Hence, $\Phi^*(\bar e_2) = \bar e_1$ and there is a commutative diagram as in \eqref{E:diag}.  By Theorem \ref{T:levine}, it remains only to demonstrate that the Euler classes $e_i \in H^2(M_i^* \bs F_i ; \Z)$, $i = 1, 2$, of the respective principal bundles $S^1 \to M_i \bs F_i \to M_i^* \bs F_i$, $i = 1, 2$, satisfy $\vphi^*(e_2) = e_1$, where $\vphi := \Phi|_{M_1^* \bs F_1}$.  

To this end, choose orientations on $M_1^*$ and on each component of $F_1$.  These choices induce an orientation on the normal bundle of each component of $F_1$ and, hence, on the fibres of the corresponding normal sphere bundles.  The diffeomorphism $\Phi$ induces orientations on $M_2^*$, the components of $F_2$ and the corresponding fibres of the normal sphere bundles, such that, in particular, the restrictions of $\Phi$ to $S^*(F_1)$ and to the normal $\sph^2$ fibres in $M_1^*$ are orientation preserving.

Since the normal sphere bundle $S^*(F_i)$, $i = 1,2$, is diffeomorphic to $F_i \x \sph^2$, there is a basis of generators $x_{i1}, \dots, x_{in} \in H^2(S^*(F_i) \; \Z) \cong \Z^n$, where each $x_{i\ell}$ is the Poincar\'e dual to the orientation class of the $\ell^\text{th}$ component of $F_i$ and corresponds to the given orientation of the fibre of $\ell^\text{th}$ component of $S^*(F)$.  By the choice of orientations and the naturality of Poincar\'e duality, it follows that $\vphi^*(x_{2\ell}) = x_{1\ell}$ for all $\ell \in \{1, \dots, n\}$.

Let $f_{i\ell} : \sph^2 \to S^*(F_i)$ denote the inclusion of a fibre of the $\ell^\text{th}$ component of $S^*(F_i)$.  By Theorem \ref{T:levine}, for each $i = 1,2$ the inclusion $\eta_i \circ f_{i\ell} : \sph^2 \to M_i^* \bs F_i$ pulls the Euler class $e_i \in H^2(M_i^* \bs F_i ; \Z)$ back to a generator $f_{i \ell}^* (\eta_i^*(e_i)) \in H^2(\sph^2; \Z)$, that is, to the Euler class of a Hopf fibration $S^1 \to \sph^3 \to \sph^2$.  In particular, the element $\eta_i^*(e_i) = \sum_{\ell = 1}^n x_{i \ell} \in H^2(S^*(F_i); \Z)$ is, with respect to the chosen orientations, the Euler class of the disjoint union of $n$ Hopf fibrations described by $S^1 \to S(F_i) \to S^*(F_i)$.

From the commutation relation $\vphi \circ \eta_1 = \eta_2 \circ \vphi$, it is now clear that
$$
\vphi^*(\eta_2^*(e_2)) = \vphi^* \left( \sum_{\ell = 1}^n x_{2 \ell} \right) 
= \sum_{\ell = 1}^n x_{1 \ell}
= \eta_1^*(e_1).
$$

As $\eta_i^* : H^2(M_i^* \bs F_i ; \Z) \to H^2(S^*(F_i); \Z)$, $i = 1,2$, is surjective, there is a canonical section $s_i : H^2(S^*(F_i); \Z) \to H^2(M_i^* \bs F_i ; \Z)$ arising from Lemma \ref{L:section}, defined in the obvious way via $s_i ( \eta_i^*(y)) := y - j_i^*(\tau_i(y))$, for $y \in H^2(M_i^* \bs F_i ; \Z)$.  Now, from
\begin{align*}
\vphi^*\circ s_2(\eta_2^*(y)) &= \vphi^*(y) - \vphi^* \circ j_2^*(\tau_2(y)) \\
&= \vphi^*(y) - j_1^* \circ \Phi^* (\tau_2(y)) \\
&= \vphi^*(y) - j_1^* \circ \tau_1 (\vphi^*(y)) \quad \text{ (by Lemma \ref{L:commute}) } \\
&= s_1(\eta_1^*(\vphi^*(y))) \\
&= s_1 \circ \vphi^* (\eta_2^*(y)),
\end{align*}
it follows that $s_1 \circ \vphi^* = \vphi^* \circ s_2$.  By combining all of the above observations one obtains
\begin{align*}
\vphi^*(e_2) &= \vphi^* \left( j_2^*(\bar e_2) + s_2(\eta_2^*(e_2)) \right) \\
&= j_1^* \left( \Phi^*(\bar e_2) \right) + s_1 (\vphi^*(\eta_2^*(e_2))) \\
&= j_1^* (\bar e_1) + s_1(\eta_1^*(e_1)) \\
&= e_1 ,
\end{align*}
as desired.

Finally, to prove Theorem \ref{T:classification}\ref{exist}, let $M^*$ be a $4$-manifold, let $\bar e \in H^2(M^*; \Z)$ and let $F \In M^*$ be a disjoint union of $n$ embedded circles.  As $M^*$ is simply connected, the Mayer--Vietoris sequence again yields a short exact sequence as in \eqref{E:MVH2} and a canonical section $\tau : H^2(M^* \bs F ; \Z) \to H^2(M^* ; \Z)$ as in Lemma \ref{L:section}.  As above, $\tau$ induces a canonical section $s :H^2(S^*(F); \Z) \to H^2(M^* \bs F ; \Z)$.  Choose a basis $x_1, \dots, x_n \in H^2(S^*(F); \Z)$ of generators corresponding to generators of the cohomology rings of the fibres of $S^*(F)$.  Let $\alpha := \sum_{\ell = 1}^n x_\ell$ and define $e := j^*(\bar e) + s(\alpha) \in H^2(M^* \bs F; \Z)$.  By Theorem \ref{T:levine}, there exists a semi-free circle action on a $5$-manifold $M$ with orbit space $M^*$ and fixed-point set $F$, such that the Euler class of the principal bundle $S^1 \to M \bs F \to M^* \bs F$ is this element $e \in H^2(M^* \bs F; \Z)$.
\end{proof}



\section{Semi-free \texorpdfstring{$S^3$}{S3}-actions on \texorpdfstring{$8$}{8}-manifolds}
\label{S:8mnfds}

In this section, a study of the restrictions placed on the topology of manifolds admitting semi-free $S^3$ actions with fixed-point components of codimension $8$ is initiated.  Such a manifold is of dimension at least $8$ and, in the case of minimal dimension, the fixed-point set consists of isolated points.  In this scenario, Theorem~\ref{T:8mnfds} demonstrates numerous obstructions to the existence of such an $S^3$ action.

Theorem \ref{T:8mnfds} will follow directly from the following lemmas.  For the sake of brevity, in each case the hypothesis that $M$ be a simply connected, closed, smooth $8$-manifold equipped with a semi-free $S^3$ action, with fixed-point set $F$ consisting of $n \in \N$ isolated points, will be suppressed.  The orbit space $M/S^3$ is a closed manifold which admits a canonical smooth structure and is, as usual, denoted by $M^*$.  By similar arguments to those in Proposition \ref{P:pi1}, it's easy to see that $M$ being simply connected is equivalent to $M^*$ being simply connected.  Hence, is a closed, smooth, simply connected $5$-manifold.  As such, it follows that $H^2(M^*; \Z)$ is free abelian, $H^j(M^*;\Z) = 0$ for $j = 1,4$, and the Betti numbers satisfy $b_2(M^*) = b_3(M^*)$.

Let $D(F) \In M$ be the union of $n$ disjoint, $S^3$-invariant $8$-disks centred at the fixed points in $M$ and let $D^*(F) \In M^*$ be its image under the quotient map $\pi : M \to M^*$.  Since the $S^3$ action is semi-free, $D^*(F)$ is a union of $n$ disjoint $5$-disks in $M^*$.

\begin{lem}
\label{L:orbitMV}
The integral cohomology groups of $M^* \bs F$ are given by
$$
H^j(M^* \bs F;\Z) \cong 
\begin{cases}
\Z, & j=0,\\
0, & j=1,5,\\
H^2(M^*;\Z), & j=2,\\
H^3(M^*;\Z), & j=3,\\
\Z^{n-1}, & j=4.
\end{cases}
$$
\end{lem}

\begin{proof}
Consider the decomposition $M^* = (M^*\bs F) \cup D^*(F)$, where $(M^*\bs F) \cap D^*(F) = D^*(F)\bs F$ is obviously homotopy equivalent to a union of $n$ disjoint copies of $\sph^4$.  Since $M^*\bs F$ is an open manifold, $H^5(M^*\bs F; \Z) = 0$.  The remaining integral cohomology groups can now be easily deduced from the Mayer--Vietoris sequence.
\end{proof}

\begin{lem}
\label{L:EulerClass}
The Euler class of the principal bundle $S^3 \to M\bs F \to M^* \bs F$ is a generator of $H^4(M^* \bs F; \Z) \cong \Z^{n-1}$ which pulls back to a generator of the principal $S^3$-bundle over each component of $M^* \bs F$.
\end{lem}

\begin{proof}
Let $p \in F \In M$ be a fixed point.  If $D(p) \In D(F)$ (resp.\ $D^*(p^*) \In D^*(F)$) denotes the component of $D(F)$ (resp.\ $D^*(F)$) containing $p$ (resp.\ $p^* = \pi(p) \in F \In M^*$), then the principal bundle $S^3 \to D(p)\bs \{p\} \to D^*(p^*)\bs \{p^*\}$ fits into the following pullback diagram of principal $S^3$-bundles
$$
\xymatrix{
S^3 \ar[d] &  S^3 \ar[d] & S^3 \ar[d] \\
D(p)\bs \{p\}  \ar[d] \ar[r]^{\hat \iota} & M \bs F \ar[d] \ar[r] & ES^3 \ar[d] \\
D^*(p^*)\bs \{p^*\} \ar[r]_{\iota} & M^* \bs F \ar[r]_\vphi & BS^3
}
$$
where $\vphi : M^* \bs F \to BS^3$ is the classifying map, and the maps $\hat \iota : D(p)\bs \{p\} \to M\bs F$ and $\iota : D^*(p^*)\bs \{p^*\} \to M^*\bs F$ are the inclusions.

Since $H^*(BS^3;\Z) = \Z[x]$, where $\deg(x) = 4$, the Euler class of any principal $S^3$-bundle is given by the pullback of $x$ under the corresponding classifying map.  By the Slice Theorem, the free action of $S^3$ on $D(p)\bs \{p\}$ is equivalent to the free, linear $S^3$ action on $\disk^8 \bs \{0\}$.  Therefore, the bundle $S^3 \to D(p)\bs \{p\} \to D^*(p^*)\bs \{p^*\}$ is homotopy equivalent to the Hopf fibration $S^3 \to \sph^7 \to \sph^4$ and, as such, its Euler class $(\vphi \circ \iota)^*(x) = (\iota)^* (\vphi^*(x))$ is a generator of $H^4(D^*(p^*)\bs \{p^*\}; \Z) \cong \Z$.  As an immediate consequence, the Euler class $\vphi^*(x) \in H^4(M^* \bs F; \Z)$ for the bundle $S^3 \to M\bs F \to M^* \bs F$ must also be a generator.
\end{proof}

\begin{lem}
\label{L:cohomM}
The integral cohomology groups of $M$ are given by
$$
H^j(M;\Z) \cong 
\begin{cases}
\Z, & j=0,8\\
0, & j=1,7,\\
H^2(M^*;\Z), & j=2,5\\
H^3(M^*;\Z), & j=3,6\\
\Z^{n-2}, & j=4.
\end{cases}
$$
In particular, the Euler characteristic of $M$ is given by $\chi(M) = n$.
\end{lem}

\begin{proof}
By Lemma \ref{L:EulerClass}, the map $H^0(M^* \bs F; \Z) \stackrel{\smile e}{\lra}  H^4(M^* \bs F; \Z)$ in the Gysin sequence for the principal $S^3$-bundle $S^3 \to M\bs F \to M^* \bs F$ must be injective and generate a $\Z$-summand of $H^4(M^* \bs F; \Z) \cong \Z^{n-1}$.  From the short exact sequence
$$
0 \to H^0(M^* \bs F; \Z) \stackrel{\smile e}{\lra}  H^4(M^* \bs F; \Z) \to H^4(M \bs F; \Z) \to 0
$$
it follows that $H^4(M \bs F;\Z) \cong \Z^{n-2}$.  All other integral cohomology groups of $M \bs F$ can be read off from the Gysin sequence, yielding
$$
H^j(M \bs F;\Z) \cong 
\begin{cases}
\Z, & j=0,\\
0, & j=1,8,\\
H^2(M^*;\Z), & j=2,5\\
H^3(M^*;\Z), & j=3,6\\
\Z^{n-2}, & j=4,\\
\Z^{n-1}, & j=7.
\end{cases}
$$

With this in hand, the Mayer--Vietoris sequence for $M$ with respect to the decomposition $(M \bs F) \cup D(F)$, where $(M \bs F) \cap D(F) = D(F) \bs F$ is homotopy equivalent to a disjoint union of $n$ copies of $\sph^7$, easily yields the cohomology groups of $M$.

In order to determine the Euler characteristic of $M$, it clearly suffices to know the integral cohomology groups.  Alternatively, as a circle subgroup $S^1 \In S^3$ must have the same fixed-point set, the Euler characteristic can be determined directly via $\chi(M) = \chi(F)$.  Either way, one obtains $\chi(M) = n$, as desired.
\end{proof}

\begin{lem}
\label{L:CharClasses}
The Pontrjagin classes, $p_1(M)$ and $p_2(M)$, the $\hat A$-genus and the signature of $M$ are trivial.  Furthermore, $\chi(M) = n$ is even.
\end{lem}

\begin{proof}
It suffices to show that $p_1(M) = 0$.  Indeed, by Corollaries 1 and 2 of \cite{mayer} (see also \cite{atiyahhirzebruch} and \cite{uchida}), any closed, smooth, simply connected $8$-manifold admitting a semi-free $S^3$ action with isolated fixed points has trivial signature and $\hat A$-genus.  If $p_1(M) = 0$, then it follows from the Signature Theorem (or the $\hat A$-genus) that $p_2(M) = 0$.  Moreover, since $\dim(M) = 8$, the intersection form is a symmetric, non-degenerate, bilinear form, hence diagonalizable (over $\R$) with $b_4(M)$ equal to the total number of eigenvalues, all of which are non-trivial.  As the signature of $M$ is trivial, the intersection form must have the same number of positive and negative eigenvalues.  Therefore, by Lemma \ref{L:cohomM}, $\chi(M) = n = b_4(M) + 2$ is even.

In order to show that $p_1(M) = 0$, consider the commutative diagram
$$
\xymatrix{
M \bs F \ar[d]_\pi \ar[r]^{\hat \jmath} & M \ar[d]^\pi \\
M^* \bs F \ar[r]_ \jmath & M^*
}
$$
where $\pi$ is the quotient map and $\hat \jmath$, $ \jmath$ are the respective inclusion maps.  Since $T(M^* \bs F) = \jmath^*(T M^*)$ and $H^4(M^*; \Z) = 0$, it follows that $p_1(M^* \bs F) = \jmath^*(p_1(M^*)) = 0$.  

On the other hand, $T(M \bs F) = \mcal V \oplus \pi^*(T(M^* \bs F))$, where $\mcal V$ is the vertical distribution of the principal bundle $S^3 \to M \bs F \stackrel{\pi}{\lra} M^* \bs F$.  In particular, $\mcal V$ is parallelizable.  As $H^4(M \bs F; \Z)$ is free abelian, the Whitney sum formula for Pontrjagin classes yields
\begin{align*}
p_1(M \bs F) 
&= p_1(\mcal V) + p_1(\pi^*(T(M^* \bs F))) \\
&= \pi^*(p_1(M^* \bs F)) = 0.
\end{align*}

Finally, from the Mayer--Vietoris sequence for $M = (M \bs F) \cup D(F)$, one sees that $\hat \jmath^* : H^4(M; \Z) \to H^4(M \bs F; \Z)$ is an isomorphism.  Together with $\hat \jmath^*(T M) = T(M \bs F)$, it then follows that $p_1(M) = (\hat \jmath^*)^{-1}(p_1(M \bs F)) = 0$, as desired.
\end{proof}

\begin{lem}
\label{L:Spin}
$M$ is spin if and only if $M^*$ is spin.
\end{lem}

\begin{proof} 
Consider again the commutative diagram from the proof of Lemma \ref{L:CharClasses}.  From the Mayer--Vietoris sequences for $M^* = (M^* \bs F) \cup D^*(F)$ and $M = (M \bs F) \cup D(F)$, it follows that $\jmath^* : H^2(M^*; \Z_2) \to H^2(M^* \bs F; \Z_2)$ und $\hat \jmath^* : H^2(M; \Z_2) \to H^2(M \bs F; \Z_2)$ are isomorphisms.  Similarly, from the Gysin sequence for the principal bundle $S^3 \to M\bs F \stackrel{\pi}{\lra} M^* \bs F$, one obtains that $\pi^* : H^2(M^* \bs F; \Z_2) \to H^2(M \bs F; \Z_2)$ is an isomorphism.

From the discussion of the tangent bundles in the proof of Lemma \ref{L:CharClasses} and the Whitney sum formula for Stiefel--Whitney classes, one obtains $w_2(M \bs F) = \hat \jmath^* w_2(M)$ and
\begin{align*}
w_2(M \bs F) &= w_2(\pi^*(T(M^* \bs F))) \\
&= \pi^* (w_2(M^* \bs F)) \\
&= \pi^* (\jmath^*(w_2(M^*))). 
\end{align*}

Therefore, $w_2(M) = (\hat \jmath^*)^{-1} \circ \pi^* \circ \jmath^* (w_2(M^*))$ and, since $(\hat \jmath^*)^{-1} \circ \pi^* \circ \jmath^* : H^2(M^*; \Z_2) \to H^2(M; \Z_2)$ is an isomorphism, it follows that $M$ is spin if and only if $M^*$ is spin.
\end{proof}

\section{Examples}
\label{S:examples}

It is perhaps illuminating to exhibit examples and constructions which illustrate the statements of Theorems \ref{T:totalspace}, \ref{T:classification} and \ref{T:8mnfds}.  This is the purpose of the present section.  See \cite{churchlamotke} for additional discussion.  The classification by Freedman of $4$-manifolds up to homeomorphism \cite{Fr} will be used freely without comment.

\begin{example}[Semi-free circle actions on $\sph^5$]
\label{E:S5}

By Theorem \ref{T:totalspace}, any semi-free circle action on $\sph^5$ with codimension-$4$ fixed-point set must fix precisely one circle in $\sph^5$ and have orbit space $M^*$ homeomorphic to $\sph^4$.  By Theorem \ref{T:classification}, there should be, up to equivariant diffeomorphism, a unique such circle action on $\sph^5$ with orbit space $M^*$.  Therefore, equivariant diffeomorphism classes of semi-free circle actions on $\sph^5$ with codimension-$4$ fixed-point set are in one-to-one correspondence with smooth manifolds homeomorphic to $\sph^4$.

To obtain the standard sphere $\sph^4$, consider $\sph^5 \In \C^3$ equipped with the semi-free circle action induced from the following action on $\C^3$:
$$
S^1 \x \C^3 \to \C^3 \,;\, (w, (x,y,z)) \mapsto w \cdot (x,y,z) = (w x, w y, z).
$$
An equivalent action is given by
$$
w \cdot (x,y,z) = (w x, \bar w y, z),
$$
where the equivalence is via the orientation-preserving diffeomorphism 
$$
\C^3 \to \C^3 \,;\, (x,y,z) \mapsto (x, \bar y, \bar z).
$$
Observe, however, that the induced diffeomorphism on the quotient pair $(\sph^4, F) = (\sph^4, \sph^1)$ preserves the orientation of $\sph^4$, while reversing that of the fixed-point set $\sph^1$.  The Euler classes of the corresponding principal $S^1$-bundles $S^1 \to \sph^5 \bs F \to \sph^4 \bs F$ will then have opposite signs.  In fact, by Theorem \ref{T:levine}, these Euler classes correspond to the generators $\pm 1$ of $H^2(\sph^4 \bs F; \Z) \cong \Z$.
\end{example}


\subsection{Semi-free circle actions on \texorpdfstring{$\sph^3$}{S3}-bundles over \texorpdfstring{$\sph^2$}{S2}}
\label{SS:bundleexamples}
{\ }

According to Theorem \ref{T:totalspace}, any semi-free circle action with codimension-four fixed-point set on either $\triv$ or $\ntriv$ must have $(n,k) \in \{(2,0), (1,1)\}$ and, hence, orbit space homeomorphic to $\sph^4$ for $k = 0$ and to $\cp^2$ for $k = 1$.  Moreover, being spin, an orbit space homeomorphic to $\sph^4$ can only occur as the orbit space of an action on $\triv$, whereas a manifold homeomorphic to $\cp^2$ is necessarily not spin and, hence, can only arise from an action on $\ntriv$.  Furthermore, Theorem \ref{T:classification} yields that, up to equivariant diffeomorphism, each manifold homeomorphic to $\sph^4$ is obtained from a unique action on $\triv$ and each manifold homeomorphic to $\cp^2$ is given by a one-parameter family of actions on $\ntriv$.

In order to provide explicit examples of semi-free circle actions on $\triv$ and $\ntriv$, it is convenient to consider $T^2$ actions on $\sss$ which have only two orbit types: orbits having circle isotropy, the union of which is a submanifold of codimension $4$, and orbits with trivial isotropy.  The manifolds $\triv$ and $\ntriv$ will be obtained as quotients of $\sss$ under a free circle subaction and inherit an induced semi-free $S^1$ action with fixed-point set of codimension $4$.  Although the orbit space $M^*$ under each such action is homeomorphic to either $\sph^4$ or $\cp^2$, it is not clear whether $M^*$, equipped with its canonical smooth structure, is diffeomorphic to $\sph^4$ or $\cp^2$, respectively.  Thus, even though we expect the examples below to comprise a complete list of semi-free circle actions on $\triv$ and $\ntriv$ with orbit space diffeomorphic to $\sph^4$ and $\cp^2$, respectively, it remains possible that exotic smooth structures are obtained on the orbit spaces.

Elements of $\sss$ will be considered as pairs $(q_1, q_2)^{\mathrm t}$ of unit quaternions, where $q_m = u_m + v_m j \in \sph^3 \In \H$ with $u_m, v_m \in \C$, $|u_m|^2 + |v_m|^2 = 1$, for $m = 1,2$.  

\begin{example}[Semi-free circle actions on $\triv$]
\label{E:spin}

Suppose that the $T^2$ action on $\sss$ is free except at points lying on the tori $\{(u_1, u_2)^{\mathrm t}\} \In (\sss) \cap \C^2$ and $\{(u_1, v_2 j )^{\mathrm t}\} \In (\sss) \cap \C \oplus \C j$, along which there is circle isotropy. Up to equivariant diffeomorphism and reparametrizations of the torus, the action of $T^2$ on $\sss$ can then be given by
$$
(w,z) \cdot \bpm q_1 \\ q_2 \epm 
=
\bpm    u_1 + w \, v_1 j \\
z  \, u_2 + z w \, v_2 j \epm.
$$

The only circle subgroups of $T^2$ which act freely on $\sss$ are $S^1_{(0,1)} := \{(1, z)\}$ and $S^1_{(-2,1)} := \{(\bar z^2, z)\}$ and their conjugates, where the action of $S^1_{(a,b)} \In T^2$ on $\sss$ is via
$$
z \cdot \bpm q_1 \\ q_2 \epm =
\bpm u_1 +  z^a v_1 j \\
z^b \, u_2 + z^{a+b} v_2 j \epm.
$$
From the Gysin sequence, together with the Barden--Smale classification \cite{barden, smale}, it is clear that the quotient $(\sss)/S^1_{(a,b)}$, $(a,b) \in \{ \pm (0,1), \pm(-2,1) \}$, under each of these free actions is either $\triv$ or $\ntriv$.  By the work of DeVito in \cite{DV} (see also \cite{GGK}, \cite{kerin}), the second Stiefel--Whitney class of $(\sss)/S^1_{(a,b)}$ is given by the mod-$2$ reduction of the sum $a + b + (a+b) = 2(a+b)$ of the exponents of $z$, that is,  $w_2((\sss)/S^1_{(a,b)}) = 0$ in each case.  Therefore, each quotient is spin and, hence, diffeomorphic to $\triv$.

The circles $S^1_{(1,0)} := \{(w,1) \} \In T^2$ and $S^1_{(a,b)}$, $(a,b) \in \{ \pm (0,1), \pm(-2,1) \}$, together generate (the homology of) the torus $T^2$.  Therefore, in each case the $T^2$ action on $\sss$ induces a semi-free circle action on $\triv$ given by
$$
w \cdot \bbm q_1 \\ q_2 \ebm =
\bbm  u_1 +  w \, v_1 j \\
u_2 + w \, v_2 j \ebm
$$
with fixed-point set $F$ consisting of the two circles $\{[u_1, u_2]^{\mathrm t} \}$ and $\{[u_1, v_2 j]^{\mathrm t} \} \In \ntriv$.  As discussed above, the orbit space $M^* = (\sss)/T^2 = (\triv)/S^1_{(1,0)}$ is homeo\-morphic to $\sph^4$.

By Theorem \ref{T:classification}, these four induced semi-free circle actions on $\triv$ are equivariantly diffeomorphic.  Moreover, by Levine \cite{levine} (see Theorem \ref{T:levine}), the Euler classes of the principal bundles $\sph^1_{(1,0)} \to (\triv) \bs F \to M^* \bs F$ can take only the values $\pm (1, \pm 1) \in H^2(M^* \bs F; \Z) \cong \Z^2$.  It is possible to pass from one to the other via diffeomorphisms of $M^*$ which interchange and/or reverse the orientation of the fixed-point components, as appropriate.
\end{example}


\begin{example}[Countably many inequivalent semi-free circle actions on $\ntriv$]
\label{E:nonspin}

Suppose now that the $T^2$ action on $\sss$ has circle isotropy at points on the torus $\{(u_1,u_2)^{\mathrm t}\} \In (\sss) \cap \C^2$ and is free elsewhere. Therefore, up to equivariant diffeo\-morphism and reparametrization of the torus, the action can be described by
$$
(w,z) \cdot \bpm q_1 \\ q_2 \epm =
\bpm z \, u_1 +  w \, v_1 j \\
z \, u_2 + z w \, v_2 j \epm.
$$
Let $M^*$ denote the orbit space $(\sss)/T^2$.  Now, for any $m \in \Z$, the circle $S^1_{(m,1)} := \{(z^m,z)\} \In T^2$ acts freely on $\sss$ via
$$
z \cdot \bpm q_1 \\ q_2 \epm =
\bpm z  \, u_1 +  z^m v_1 j \\
z \, u_2 + z^{m+1} v_2 j \epm.
$$
As in the previous example, it is clear that the quotient $(\sss)/S^1_{(m,1)}$ under this free action is either $\triv$ or $\ntriv$ and, again by \cite{DV}, the second Stiefel--Whitney class of $(\sss)/S^1_{(m,1)}$ is given by the mod-$2$ reduction of the sum $1 + m + 1 + (m+1) = 2m + 3$, that is,  $w_2((\sss)/S^1_{(m,1)}) = 1$.  Thus, not being spin, the quotient is $\ntriv$ for every $m \in \Z$.

The circles $S^1_{(1,0)} := \{(w,1) \} \In T^2$ and $S^1_{(m,1)}$ together generate (the homology of) the torus $T^2$.  Therefore, for each $m \in \Z$, the $T^2$ action on $\sss$ induces a semi-free circle action on $\ntriv$ given by
$$
w \cdot \bbm q_1 \\ q_2 \ebm =
\bbm  u_1 +  w \, v_1 j \\
u_2 + w \, v_2 j \ebm
$$
and fixing the circle $F = \{[u_1, u_2]^{\mathrm t} \} \In \ntriv$.  As discussed above, the orbit space $M^* = (\sss)/T^2 = (\ntriv)/S^1_{(1,0)}$ is homeomorphic to $\cp^2$.  However, it will turn out that these induced semi-free actions are mutually inequivalent. Indeed, if $e \in H^2(M^* \bs F; \Z) \cong H^2(M^*; \Z) \oplus \Z \cong \Z^2$ denotes the  Euler class of the principal bundle $S^1_{(1,0)} \to (\ntriv) \bs F \to M^* \bs F$ then, as a consequence of Levine's work \cite{levine} (see Theorem \ref{T:levine}), it has the form $e = (e_m, \pm 1)^{\mathrm t}$, where $e_m \in H^2(M^*; \Z)$ possibly differs from the canonical element indicated by Theorem \ref{T:classification} due to the choice of splitting.  It will be shown below that $e_m$ is linear in the parameter $m \in \Z $ (so that the inequivalence of the actions actually follows Theorem \ref{T:levine}).

For convenience, let $P = (\sss) \bs \{(u_1, u_2)^{\mathrm t}\} \In \sss$ denote the preimage of $(\ntriv) \bs F$ under the quotient map $\sss \to (\sss)/S^1_{(m,1)} \cong \ntriv$.   Since $S^1_{(m,1)}$ is a normal subgroup of $T^2$, the quotient homomorphism $\rho_m : T^2 \to T^2/S^1_{(m,1)} \cong S^1_{(1,0)}$ induces a map $B\rho_m : BT^2 \to B(T^2/S^1_{(m,1)}) \cong BS^1$ on classifying spaces, such that the principal bundle $S^1_{(1,0)} \to (\ntriv) \bs F \to M^* \bs F$ is the pullback of the universal bundle under the map $B \rho_m \circ \vphi : M^* \bs F \to B S^1$, where $\vphi : M^* \bs F \to B T^2$ is the classifying map of the principal bundle $T^2 \to P \to P/T^2 \cong M^* \bs F$.  Hence, the Euler class $e = (e_m, \pm 1)^{\mathrm t}$ is, up to sign, the pullback of a generator of $H^2(BS^1; \Z) \cong \Z$ under the induced map $\vphi^* \circ (B \rho_m)^* : H^2(B S^1; \Z) \to H^2(M^* \bs F; \Z)$.

With respect to the basis given by the product decomposition $T^2 \cong S^1_{(1,0)} \x S^1_{(0,1)}$, the map $i_* : \pi_1(S^1_{(m,1)}) \to \pi_1(T^2)$ induced by the inclusion $i : S^1_{(m,1)} \to T^2$ can be written as $i_* = (m,1)^{\mathrm t} : \Z \to \Z^2$.  From the long exact sequence for the fibration $S^1_{(m,1)} \to T^2 \to T^2 / S^1_{(m,1)}$ it then follows that $(\rho_m)_* : \pi_1(T^2) \to \pi_1(T^2 / S^1_{(m,1)})$ can be written as $(\rho_m)_* = \ve (1, -m) : \Z^2 \to \Z$, for some $\ve \in \{ \pm 1\}$, which in turn implies that $(B \rho_m)^* : H^2(BS^1; \Z) \to H^2(BT^2)$ is given by $(B \rho_m)^* = \ve (1, -m)^{\mathrm t} : \Z \to \Z^2$ (with respect to the induced basis).

With respect to a fixed basis of $H^2(M^* \bs F ; \Z)$ given by the direct-sum decomposition $H^2(M^* \bs F; \Z) \cong H^2(M^*; \Z) \oplus \Z \cong \Z^2$, there is, since $P$ is $2$-connected, a matrix $\left(\bsm a & b \\ c & d \esm\right) \in \SL_2(\Z)$ describing the isomorphism $\vphi^* : H^2(BT^2 ; \Z) \to H^2(M^* \bs F ; \Z)$ induced from the classifying map $\vphi : P/T^2 \cong M^* \bs F \to BT^2$.  Consequently, the Euler class has the form
$$
e = \bpm e_m \\ \pm 1 \epm
=  \ve \bpm a & b \\ c & d \epm \bpm 1 \\ -m \epm
= \ve \bpm a - b m \\ c - d m \epm \in H^2(M^* \bs F; \Z).
$$ 
However, since $\vphi^* \in \SL_2(\Z)$ and the second entry of $e$ is independent of $m$, it may be concluded that $b = c = \pm \ve$ and $d = 0$, hence, that 
$$
e = \bpm e_m \\ \pm 1 \epm 
= \bpm \ve a \mp m \\ \pm 1 \epm \in H^2(M^* \bs F; \Z).
$$
Observe, finally, that the diffeomorphism group $\Diff(M^*)$ can, at worst, change the sign of $e_m \in H^2(M^*; \Z)$.  Therefore, by Theorem \ref{T:classification}, this countable family of pairwise-inequivalent semi-free circle actions on $\ntriv$ with fixed-point set of codimension four comprises all such actions on $\ntriv$ with orbit space $M^*$, up to equivariant diffeo\-morphism.
\end{example}

\begin{rem}
Note that, while the examples of semi-free circle actions on $\triv$ and $\ntriv$ are induced by $\T^2$ actions on $\sss$, semi-free circle actions on connected sums of $\triv$ and $\ntriv$ will not lift to $T^2$ actions on a connected sum of copies of $\sss$ in a similar way.  Indeed, by the Gysin sequence, no $2$-connected $6$-manifold can be the total space of a principal $S^1$-bundle over a connected sum of copies of $\triv$ and $\ntriv$ having at least two summands.
\end{rem}


\subsection{Semi-free circle actions on connected sums}
\label{SS:connected sums}
{\ }

For $i = 1, 2$, let $M_i$ be a $5$-manifold admitting a semi-free circle action with fixed-point set $F_i \In M_i$ of codimension $4$.  Let $n_i \in \N$ be the number of components of $F_i$ and let $p_i \in F_i$ be a fixed point.  In particular, the isotropy representation of $S^1$ on $T_{p_1} M_1$ is equivalent to that on $T_{p_2} M_2$, since the action on the unit tangent $4$-sphere in each case is simply the suspension of the Hopf action on $\sph^3$.

Therefore, it is possible to perform an equivariant connected sum by removing $S^1$-invariant neighbourhoods (disks) of $p_1 \in M_1$ and $p_2 \in M_2$ and gluing along the boundaries via an equivariant diffeomorphism (being careful with orientations).  In so doing, one obtains a semi-free $S^1$ action on $M = M_1 \# M_2$ with codimension-$4$ fixed-point set consisting of $n = n_1 + n_2 - 1$ components.  To see this, note that removing $S^1$-invariant disks around $p_1$ and $p_2$ involves removing arcs from the component (circle) of $F_1$ and $F_2$ containing each point.  The equivariant-connected-sum operation then glues the remaining parts of these two circles along their boundaries, thus forming a single fixed circle in $M$.  The orbit space $M^*$ of this semi-free action on $M$ is clearly seen to be $M_1^* \# M_2^*$.  Recall, moreover, that the behaviour of Stiefel--Whitney classes of positive degree is well understood and, in particular, $w_2(M) = (w_2 (M_1), w_2(M_2)) \in H^2(M; \Z_2) \cong H^2(M_1; \Z_2) \oplus H^2(M_2; \Z_2)$, that is, $M$ is spin if and only if both $M_1$ and $M_2$ are spin.

By taking equivariant connected sums of the manifolds given in Examples \ref{E:spin} and \ref{E:nonspin}, one obtains examples of semi-free circle actions on all of the spaces given in Theorem \ref{T:totalspace}.

\begin{prop}
\label{P:connsum}
	For any $m_1, m_2 \in \N \cup \{0\}$, a connected sum $M$ of $m_1$ copies of $\triv$ and $m_2$ copies of $\ntriv$ admits a semi-free circle action with fixed-point set consisting of $m_1 + 1$ circles.  Moreover, $M$ is spin if and only if $m_2 = 0$.
	\end{prop}

\begin{proof}
	If $m_1 = m_2 = 0$, then Example \ref{E:S5} gives the desired space and action.  If $m_1 + m_2 > 0$ then, by taking an equivariant connected sum of $m_1$ copies of Example \ref{E:spin} and $m_2$ copies of Example \ref{E:nonspin}, one obtains a manifold $M$ and action as in the statement. That $M$ is spin if and only if $m_2 = 0$ follows directly from the above remarks on Stiefel--Whitney classes. 
\end{proof}

Let $M = M_1 \# M_2$ be an equivariant connected sum as described above.  The usual map from $M$ to the wedge sum $M_1 \vee M_2$ induces an isomorphism on cohomology in degree two and, by construction, maps the fixed-point set $F \In M$ to $F_1 \vee F_2 \In M_1 \vee M_2$, where $F_1 \vee F_2$ denotes the wedge sum of a single component of $F_1$ with a single component of $F_2$.  Now, the image of $M  \bs F$ in $M_1 \vee M_2$ is clearly the disjoint union of $M_1 \bs F_1$ with $M_2 \bs F_2$.  It then follows easily that the canonical element $\bar e \in H^2(M^*; \Z)$ from Theorem \ref{T:classification} which characterizes the semi-free circle on $M$ with fixed-point set $F$ and orbit space $M^*$ is given by $(\bar e_1, \bar e_2) \in H^2(M_1; \Z) \oplus H^2(M_2; \Z)$, where, for $i = 1,2$, $\bar e_i \in H^2(M_i; \Z)$ is the corresponding canonical element characterizing the semi-free action on $M_i^*$ with fixed-point set $F_i$ and orbit space $M_i^*$.

Observe, finally, that the examples constructed by taking equivariant connected sums as above cannot exhaust all possible actions.  Indeed, by \cite{levine} (see Theorem \ref{T:levine}), all $4$-manifolds arise as the orbit space of a semi-free circle action on a $5$-manifold, such that the fixed-point set has codimension four.  However, taking equivariant connected sums of Examples \ref{E:S5}--\ref{E:nonspin} yields only orbit spaces which are homeomorphic to connected sums of copies of $\cp^2$.  In particular, an orbit space homeomorphic to $\sph^2 \x \sph^2$ is never achieved.  This will be remedied in the next class of examples.


\subsection{Semi-free circle actions on fibre sums}
\label{SS:fibresums}
{\ }

Suppose two $5$-manifolds $M_1$ and $M_2$ are equipped with arbitrary effective circle actions.  Although the orbit spaces $M_1^*$ and $M_2^*$ will not, in general, be topological manifolds, it is well known that the principal isotropy subgroup of each action is trivial.  Then, by the Slice Theorem, for each $i = 1,2$ an $S^1$-invariant $\ve$-neighbourhood $\nu_\ve (S^1 \cdot p_i)$ of a principal orbit $S^1 \cdot p_i \In M_i$ is equivariantly diffeomorphic to $\sph^1 \x \disk^4$ (with $S^1$ acting trivially on the second factor).

By reversing the orientation of the $S^1$ action on $M_2$ if necessary, one can therefore glue the complements of two such invariant neighbourhoods via an orientation-reversing diffeo\-morphism of their boundaries to obtain a new manifold
$$
M = (M_1 \bs \nu_\ve (S^1 \cdot p_1)) \cup_\partial (M_2 \bs \nu_\ve (S^1 \cdot p_2))
$$ 
on which $S^1$ also acts effectively, thus obtaining the \emph{equivariant fibre sum} $M$.  In particular, notice that the quotient of $\nu_\ve(S^1 \cdot p_i) \In M_i$, $i = 1,2$, under the respective circle action is diffeomorphic to $\disk^4$, from which it follows that the orbit space $M^*$ of the induced circle action on $M$ is given by $M_1^* \# M_2^*$.

Suppose now that the circle action on $M_1$ is free and that the action on $M_2$ is semi-free with fixed-point set $F$ of codimension four.  Then the induced action on the equivariant fibre sum $M$ is also semi-free with fixed-point set $F$ and orbit space the topological manifold $M^* = M_1^* \# M_2^*$ (which admits a canonical smooth structure).

\begin{prop}
\label{P:fibresums}
If $B$ is a closed, smooth, simply connected $4$-manifold, then, for every $n \in \N$, there exists a closed, smooth, simply connected $5$-manifold $M$ on which $S^1$ acts smoothly and semi-freely with fixed-point set $F$ consisting of $n$ circles and orbit space $M^*$ homeomorphic to $B$.
\end{prop}

\begin{proof}
Recall from Example \ref{E:S5} and Proposition \ref{P:connsum} that $M_1 = \#_{j=1}^{n-1} (\triv)$ admits a semi-free circle action with fixed-point set $F$ consisting of $n$ circles.  Let, on the other hand, $M_2$ be the total space of a principal $S^1$-bundle over $B$.  Then the equivariant fibre sum $M$ of $M_1$ and $M_2$ admits a semi-free circle action with fixed-point set $F$ and orbit space homeomorphic to $\sph^4 \# B$, that is, homeomorphic to $B$.
\end{proof}

Notice that, by Theorem \ref{T:totalspace} (and Proposition \ref{P:pi1}), the diffeomorphism type of the manifold $M$ in Proposition \ref{P:fibresums} is determined by whether $B$ is spin or not.  Suppose now that $M$ is an $(m+1)$-fold connected sum $M$ of $\sph^3$-bundles over $\sph^2$, for some $m \in \N$.  Then, if $M$ is not spin, Proposition \ref{P:connsum} implies that $M$ admits a semi-free circle action with fixed-point set $F$ consisting of $n$ circles, for all $n \in \{1, \dots, m\}$.  

On the other hand, if $M$ is spin, then $M^*$ is a smooth, spin $4$-manifold.  By Rokhlin's Theorem \cite{rokhlin}, the intersection form of $M^*$ must have signature divisible by $16$ and, hence, $b_2(M^*)$ must be even.  Then Theorem \ref{T:totalspace} and Propositions \ref{P:connsum} and \ref{P:fibresums} together imply that $M$ admits a semi-free circle action with fixed-point set $F$ consisting of $n$ circles, for all $n \in \{ m - 2(j - 1) \mid 0 \leq j \leq \lfloor \frac{m+1}{2} \rfloor \}$.


\subsection{Semi-free \texorpdfstring{$S^3$}{S3} actions in dimension \texorpdfstring{$8$}{8}}
\label{SS:dim8egs}
{\ }

Without a good understanding of $8$-manifolds, it seems very difficult to construct explicit examples of $8$-manifolds admitting explicit semi-free $S^3$ actions with isolated fixed points, even though Church and Lamotke \cite{churchlamotke} provide a algorithm for constructing such spaces over any given orbit space.  Nevertheless, the following examples should give a good impression of what is to be expected.

Suppose $M$ is an $8$-manifold with $H^2(M; \Z) = H^3(M; \Z) = 0$ which admits a semi-free $S^3$ action with only isolated fixed points.  Then the orbit space $M^*$, equipped with its canonical smooth structure, is diffeomorphic to $\sph^5$.  Indeed, from Theorem \ref{T:8mnfds}, it follows that $H^*(M^*;\Z) = H^*(\sph^5;\Z)$.  Now, by collapsing the complement of a neighbourhood of a point in $M^*$, one obtains a map $f:M^* \to \sph^5$ of degree $1$.  By the Hurewicz and Whitehead theorems, it follows that $M^*$ is a homotopy $5$-sphere.  By the $h$-cobordism theorem, $M^*$ is homeomorphic to $\sph^5$ and, as there are no exotic spheres in dimension $5$, one concludes that $M^* \cong \sph^5$.

Taking advantage of the Hopf action, such semi-free actions are, in fact, easy to find:
\begin{enumerate}
\item The standard sphere $\sph^8 \In \R \x \HH^2$ can be written as
$$
\sph^8 = \{(s,x_1,x_2)^{\mathrm t} \mid s \in \R, \ x_1, x_2 \in \HH,\ s^2 + |x_1|^2 + |x_2|^2 = 1 \} .
$$
Then $S^3$ acts semi-freely on $\sph^8$ via
$$
q \cdot \bpm s \\ x_1 \\ x_2 \epm = \bpm s \\ q \, x_1 \\  q \, x_2 \epm
$$
with fixed-point set $F = \{(\pm 1, 0, 0)^{\mathrm t} \} \In \sph^8 $ consisting of $2 = \chi(\sph^8)$ points.

\item The standard $\sph^4 \In \R \x \HH$ is given by
$$
\sph^4 = \{ (s,x)^{\mathrm t} \mid s \in \R, \ x \in \HH,\ s^2 + |x|^2 = 1 \}.
$$
The action of $S^3$ on $\sph^4 \x \sph^4$ described by 
$$
q \cdot \left( \bpm s_1 \\ x_1 \epm, \bpm s_2 \\ x_2 \epm \right) 
= \left( \bpm t_1 \\ qx_1 \epm, \bpm t_2 \\ qx_2 \epm \right)$$
is semi-free with $4 = \chi(\sph^4 \x \sph^4)$ fixed points.
\item Suppose that, for $j = 1,2$, $M_j$ is an $8$-manifold admitting a semi-free $S^3$ action with fixed-point set consisting of $n_j = \chi(M_j)$ points.  Then, by removing (as for semi-free circle actions) an $S^3$-invariant neighbourhood of a fixed point in each space and gluing equivariantly along the boundaries, one obtains an equivariant connected sum $M = M_1 \# M_2$, such that the resulting semi-free $S^3$ action on $M$ has $n_1 + n_2 - 2 = \chi(M)$ fixed points.  Applying this construction to multiple copies of the previous example yields, for each $m \in \N$, a semi-free $S^3$ action on $\#_{j=1}^m (\sph^4 \x \sph^4)$ with $2m+2$ fixed points.
\end{enumerate}

As mentioned in the introduction, the work of Church and Lamotke \cite{churchlamotke} ensures that this is a complete classification up to equivariant homeomorphism of all semi-free actions on $8$-manifolds having  isolated fixed points and $\sph^5$ as orbit space.

To obtain more interesting orbit spaces, one can use an equivariant-fibre-sum construction analogous to that for semi-free circle actions.  This is essentially the construction suggested by \cite{churchlamotke}.  In the present case, one considers a principal bundle $S^3 \to P \to B$ over a $5$-manifold $B$ and an (equivariant) connected sum $\#_{j=1}^m (\sph^4 \x \sph^4)$, with $m \in \N \cup \{0\}$.  Removing an $S^3$-invariant neighbourhood of a principal orbit from each and gluing along the boundaries yields a manifold $M$ which admits a semi-free $S^3$ action with $2m+2$ fixed points and orbit space $M^*$ diffeomorphic to $\sph^* \# B \cong B$.  Moreover, up to equivariant homeomorphism, all actions are obtained in this way.

Without further information, however, this is a rather unsatisfactory construction.  Indeed, it yields an explicit description of neither the manifold $M$ nor the semi-free $S^3$ action upon it.  For example, it is unknown whether $M = (\sph^2 \x \sph^6) \# (\sph^3 \x \sph^5)$ admits a semi-free $S^3$ action with isolated fixed points, even though Theorem \ref{T:8mnfds} ensures that the individual summands do not.  Theorem \ref{T:8mnfds} and \cite{barden} would together imply that the orbit space of such an action is $M^* = \sph^3 \x \sph^2$ and that the number of fixed points is $2 = \chi(M)$.  By the previous paragraph, $M$ would then have to be (equivariantly) homeomorphic to the fibre sum of $\sph^8$, equipped with its semi-free action with orbit space $\sph^5$, and the trivial principal $S^3$-bundle $\sph^3 \x \sph^5$.


\end{document}